\documentclass[11pt,fullpage, reqno]{amsart}
\usepackage[pagewise]{lineno}
\usepackage{amssymb,amsthm,amsmath}
\usepackage{tikz}
\usetikzlibrary{arrows.meta,decorations.markings}
\usepackage{mathtools}
\usepackage{amsfonts}
\usepackage{amscd}
\usepackage{amssymb}
\usepackage{enumerate}
\usepackage{bbm}
\usepackage{graphicx}
\allowdisplaybreaks
\usepackage{mathabx}
\usepackage{color}
\usepackage{amsbsy}
\usepackage{graphicx}
\usepackage{amsthm}
\usepackage{amsmath}
\usepackage{amsxtra}
\usepackage{mathrsfs}
\usepackage{bbm}
\usepackage{dsfont}
\usepackage{enumitem}
\usepackage{comment}
\usepackage{relsize}
\usepackage{enumerate}
\usepackage{xcolor}
\usepackage{float} 
\usepackage{url}
\usepackage{soul}
\usepackage[hidelinks]{hyperref} 
\usepackage{ifthen}

\newtheorem{theorem}{Theorem}[section]
\newtheorem{lemma}[theorem]{Lemma}
\newtheorem{corollary}[theorem]{Corollary}
\newtheorem{proposition}[theorem]{Proposition}

\theoremstyle{definition}

\theoremstyle{remark}
\newtheorem{Remark}[theorem]{\rm \bf Remark}
\numberwithin{equation}{section}

\newcommand{\C}{\mathbb{C}}

\newcommand{\R}[1]{\mathbb{R}^{#1}}
\newcommand{\bN}{\mathbb{N}}

\newcommand{\x}{\mathbf{x}}

\newcommand{\fa}{\mathfrak{a}}
\newcommand{\fg}{\mathfrak{g}}
\newcommand{\fh}{\mathfrak{h}}
\newcommand{\fk}{\mathfrak{k}}

\newcommand{\fn}{\mathfrak{n}}

\newcommand{\fp}{\mathfrak{p}}
\newcommand{\fl}{\mathfrak{l}}
\newcommand{\fq}{\mathfrak{q}}

\newcommand{\aq}{\mathfrak{a}_q}

\newcommand{\cF}{\mathcal{F}}
\newcommand{\cW}{\mathcal{W}}
\newcommand{\cO}{\mathcal{O}}
\newcommand{\cI}{\mathcal{I}}
\newcommand{\cJ}{\mathcal{J}}

\newcommand{\sS}{\mathscr{S}}

  \newcommand{\on}{\overline{\mathfrak{n}}}

\newcommand{\beas}{\begin{eqnarray*}}
\newcommand{\eeas}{\end{eqnarray*}}
\newcommand{\bes} {\begin{equation*}}
\newcommand{\ees} {\end{equation*}}
\newcommand{\be} {\begin{equation}}
\newcommand{\ee} {\end{equation}}
\newcommand{\bea} {\begin{eqnarray}}
\newcommand{\eea} {\end{eqnarray}}

\let\oldproofname=\proofname
\renewcommand{\proofname}{\rm\bf{\oldproofname}}

\newcommand{\eE}{E^{\circ}}

\renewcommand{\l}{\lambda}

\newcommand{\mr}[1]{\mathrm{#1}}

\title{$L^r$- Schwartz spaces on split rank one semisimple symmetric spaces}
\author{Sanjoy Pusti and Iswarya Sitiraju}
\address{Sanjoy Pusti \endgraf Department of Mathematics, \endgraf INDIAN INSTITUTE OF TECHNOLOGY BOMBAY,
\endgraf Powai, Mumbai-400076, India.}
\email{sanjoy@math.iitb.ac.in}

\address{Iswarya Sitiraju, \endgraf Department of Mathematics, \endgraf INDIAN INSTITUTE OF TECHNOLOGY BOMBAY,
\endgraf Powai, Mumbai-400076, India.}
\email{iswarya@math.iitb.ac.in}


\begin{document}
\begin{abstract}
    We study the left $K$-invariant $L^r$-Schwartz space and its Fourier transform on split rank one semisimple symmetric spaces $G/H$ for $0<r\leq 2$. We explicitly determine the kernel of the Fourier transform, showing that it is spanned by eigenfunctions corresponding to the discrete spectrum of the Laplace–Beltrami operator on $G/H$.
\end{abstract}

\maketitle

\section{Introduction}

In the classical Euclidean theory on $\R{n}$, the Fourier transform $\cF$ maps a smooth function with rapid decay to a smooth function with rapid decay. More precisely, it is an isomorphism of the Schwartz space $\sS(\R{n})$ onto itself. The inverse Fourier transform is explicitly given and the Fourier transform extends to a unitary operator on $L^2(\R{n})$ yielding the Plancherel theorem. This theory was extended in the non abelian settings with the works of Harish Chandra, who developed the theory of harmonic analysis on the real reductive Lie groups. In particular, Harish Chandra  introduced the notions of Schwartz space, Eisenstein integrals, and wavepackets, and established the Plancherel theorem in a series of papers \cite{HC75,HC76a,HC76b} in 1970's. \\

Harish Chandra also developed the theory of spherical functions in \cite{HC58}. These functions are an analogue of exponential functions on $\R{n}$. Let $G$ be a semisimple linear Lie group with finite center and let $K$ be a maximal compact subgroup of $G$. For spectral parameter $\l$, the spherical functions $\phi_\l$ are matrix coefficients of spherical representations $\pi_\l$ with respect to the $K$-fixed vectors. A complete treatment of Harmonic analysis on Riemannian symmetric space $G/K$ is given in the books by Helgason \cite{He78, He84}. 
  In the Riemannian setting, $(G,K)$ is a Gelfand pair. The work of van Dijk and his students develop the harmonic analysis associated with Gelfand pairs and generalized Gelfand pairs, including the spherical transform and Plancherel theorem (see \cite{vD86,vanDijk2009}).\\

Initial work in the direction of pseudo Riemannian symmetric spaces $G/H$ (where $H$ need not be compact) was carried out on hyperbolic spaces in late 1970's and 1980's, including the works of Rossman \cite{R78}, Strichartz \cite{St73}, Molchanov \cite{M81}, Faraut \cite{F79}, and Shintani \cite{Sh67}. 
Moreover, the harmonic analysis for a comprehensive list of rank one connected non-Riemannian symmetric spaces (up to covering) has been studied (see \cite{vD86,Molchanov1986} and references therein). These include (the list of references is not exhaustive):
\begin{enumerate}
\item The real hyperbolic spaces: $\mathrm{\mathrm{SO}}_e(p, q)/\mathrm{\mathrm{SO}}_e(p-1, q), p>1, q>0$ (see \cite{Gelfand1966,Limim1967,M81,F79,St73,Sh67,R78, Molchanov1986});
\item The complex hyperbolic spaces: $\mathrm{SU}(p, q)/S\left(\mathrm{U}(1)\times\mathrm{U}(p-1, q)\right), p>1, q>0$ (see \cite{F79,Matsumoto1978,Molchanov1986});
\item The quaternion hyperbolic spaces: $\mathrm{Sp}(p, q)/\left(\mathrm{Sp}(1)\times\mathrm{Sp}(p-1, q)\right), p>1, q>0$ (see \cite{F79,Molchanov1986});
\item The octonion hyperbolic space: $F_{4(-20)}/Spin(1, 8)$ (see \cite{Kosters1983,Molchanov1986});
\item $\mathrm{SL}(n, \mathbb R)/\mathrm{GL}_+(n-1, \mathbb R)$, $n>1$ (see \cite{VanDijkPoel1984,Molcanov1982,VanDijkPoel1984, M83,Molchanov1986});
\item $\mathrm{Sp}(n, \mathbb R)/\left(\mathrm{Sp}(1, \mathbb R)\times\mathrm{Sp}(n-1, \mathbb R)\right), n>1$ (see \cite{Kosters1985,Molchanov1986});
\item $F_4(4)/Spin(4, 5)$ (see \cite{Kosters1985,Molchanov1986}).
\end{enumerate} 
The theory of decomposing $L^2$- functions on a wider class of symmetric spaces, known as $K_\epsilon$ spaces, was studied by Oshima and Sekiguchi in \cite{OS80}  for the continuous spectrum. The discrete part in the decomposition of $L^2(G/H)$ is understood from the works of Flensted-Jensen \cite{F80} and Oshima and Matsuki \cite{OM84}. The work on establishing Plancherel theorem for general reductive symmetric spaces was developed further by van den Ban in \cite{vDB88,vdB92} and also in collaboration with Schlichtkrull in \cite{VS97-3}. Simulateneously, Delorme developed the theory of generalized principal series representations that occur in the decomposition of $L^2(G/H)$ independently in \cite{BrylinskiDelorme1992,CarmonaDelorme1994, CD98, Delorme1994, Delorme1997} with Brylinski and Carmona. In contary, the Plancherel and Paley Wiener theorem was proved by establishing Fourier inversion formula by van den Ban and Schlichtkrull \cite{vS99,vdS05, VS06,vDS05-2}. For a complete overview of the general theory see \cite{AO05, HS94} and references therein.\\

Building on Harish-Chandra's work on $L^2$ Schwartz space of $K$ biinvariant functions, the $L^r$ Schwartz spaces of $K$ biinvariant functions was first studied by Trombi and Varadarajan in \cite{TrombiVaradarajan1971} for $0<r<2$. A shorter proof was later given by Anker in \cite{Anker1991} using Abel transform. The $L^r$ Schwartz space isomorphism theorem for functions on $G/K$ was addressed in \cite{EguchiKowata1977} (also see \cite{GV88}). The Plancherel theorem including the Fourier inversion formula for $L^2$ -Schwartz space $\mathcal{C}^2(G/H)$ which are $K$-finite was also established by van den Ban and Schlichtkrull in \cite[Thm 21.2]{vdS05} and independently by Delorme in \cite[Thm. 2]{D99}. In the case of pseudo Riemannian real hyperbolic spaces $SO_e(p,q)/SO_e(p-1,q)$ for $p,q>1$ , the Fourier transform on $L^r$ Schwartz space and its image were studied in \cite{A01} for fixed $K$-type. \\

Let $G/H$ be a connected symmetric space of split rank one and $K$ be a maximal compact subgroup of $G$. Let $(\mathfrak{a}_q,\Sigma,m)$ be the root system associated with $(G,H,K)$ and $W$ be the Weyl group corresponding to the root system $(\mathfrak{a}_q,\Sigma)$. In the split rank one cases, $\fa_q \cong \R{}$. Let $\cW$ be a specific subgroup of the Weyl group $W$ that arises from the structure of $(G,H,K)$. For a left $K$-invariant functions on $G/H$, the Fourier transform is defined using the Eisenstein integrals $E(\l,\eta)(\cdot)$ for $\l \in \C$ and $\eta \in \C^\cW$. The normalized Eisenstein integrals $\eE$ are defined as the  matrix coefficients of $K$-finite vectors with $H$-fixed vectors of parabolically induced representations with certain normalization of $E$. These functions were first introduced in \cite{vdB92} and also studied in \cite{VS97,VS97-2}. The left $K$ invariant Eisenstein integrals are the generalizations of the elementary spherical functions $\phi_\l$ on the Riemannian symmetric spaces that is, when $H=K$ (cf. \cite{HC58}) and the spherical functions introduced by Oshima and Sekiguchi for the $K_\epsilon$ symmetric spaces introduced in \cite{OS80}.  Unlike the Riemannian case, $ \l \mapsto E(\l,\cdot)(\cdot)$ is a meromorphic function. For left $K$-invariant normalized Eisenstein integral $\eE$ on the split rank one symmetric spaces, the Helgason-Johnson's type theorem and Hausdorff-Young's inequality were studied while tackling the poles of $\eE$ in \cite{PS25}. 
\\

In this article, we study the left $K$-invariant $L^r$ Schwartz space on $X = G/H$, denoted by $\mathcal{C}^r(G/H)^K$, for $0 < r \leq 2$ and their Fourier transform. For a fixed $0 < \epsilon_0 < 1/2$, consider the strip \[S_r := \{\l \in \C: -\gamma_r\leq \Re\l\leq \epsilon_0\}=i\R{} + [-\gamma_r,\epsilon_0],\]
where 
\[\gamma_r = \left(\frac{2}{r} -1\right)\rho.\]
Then the Schwartz  space $\sS(S_r)_e$ on the spectral side is defined as the space of certain $\C^\cW$-valued meromorphic functions on $S_r^o$. We show that the Fourier transform $\cF$ maps $\mathcal{C}^r(G/H)^K$ to $\sS(S_r)_e$ surjectively. We extend the work of Andersen \cite{A01} to the split rank one semisimple symmetric spaces which includes the space $SO_e(p,1)/SO_e(p-1,1)$ when, referring to the list above, $q = 1$. Moreover, we explicitly determine the kernel of the Fourier transform. Let $L$ denote the set of finitely many simple poles of $\eE$ for $\Re \l >0$ given by
\[L = \{\l_k:= \rho -m_1^+-m_2^+-1-2k: k \in \bN_0\,\text{and} \, \rho -m_1^+-m_2^+-1-2k>0\},\]
where the constants $\rho, m_1^+,m_2^+$ are described in Section 2, which depend on the structure of $(G,H,K)$. 
Define
\[\Phi_{-\l_k}^0(t) = (2\cosh t)^{-\l_k-\rho}{}_2F_1\left(\frac{\rho+\l_k}{2}, -k; 1+\l_k; \cosh^{-2}t \right),\]
for $\l_k \in L$. We set
\bes L_r = \{\l \in L: \l  < \gamma_r\} \text{ and } L_r^c=\{\l\in L: \l>\gamma_r\},
\ees for $0 < r\leq 2$. Then 
\[L_r = L \text{ and } L_r^c=\emptyset\]
for $0<r<1$. Below is our main theorem (see Section 3 for the notations):
\begin{theorem}\label{thm:L-r-schwartz}  Let $0 < r \leq 2$ such that $\gamma_r$ is not a singularity of $\eE$. 
\begin{enumerate}
    \item  The Fourier transform $\cF: \mathcal{C}^r(X)^K \longrightarrow \mathscr{S}(S_r)_e$ is a continuous linear  map.
    \item The map $\cI_r: \sS(S_r)_e \longrightarrow \mathcal{C}^r(X)^K$ is a continuous linear  map. 
    \item Moreover, the operator $\cF \cI_r$ is an identity operator on $\sS(S_r)_e$.
    \item Let $r \geq 1$,
    \begin{enumerate}
        \item If $L$ is non empty then the kernel of $\cF$ is the space spanned by the functions $\{\Phi^0_{-\l_k}(t): \l_k \in L_r^c\}$.
        \item If $L$ is empty, then the Fourier transform $\cF$ is an injective linear  map.
    \end{enumerate} 
    \item If $0<r<1$, then the Fourier transform $\cF$ is an injective linear  map.
\end{enumerate}
    
\end{theorem}

To prove the theorem we adapt the main ideas of the proof in \cite{A01} and \cite{TrombiVaradarajan1971}. Our main references for this article are \cite{A01,AO05,HS94}. Throughout this article, we denote $\bN_0 = \{0,1,\ldots\}$
and $X = G/H$.

\subsection*{Acknowledgements} We would like to express our sincere gratitude to Gestur \'Olafsson for his guidance and valuable suggestions.

\section{Split rank one symmetric spaces}
Let $G$ be a connected semisimple linear Lie group, $H$ be a connected symmetric subgroup of $G$ and $K$ a maximal compact subgroup of $G$. Let $\theta$ be the Cartan involution and $\sigma$ be a involution such that $H=G^\sigma_e,$ where $G^\sigma$ is the subgroup of fixed points for $\sigma$ and $G^\sigma_e$ is the identity component of $G^\sigma$. We  assume that $\sigma\theta=\theta\sigma$. Let \bes \mathfrak{g}=\mathfrak{h}\bigoplus \mathfrak{q}, \ees be the decomposition of $\mathfrak{g}$ induced by $\sigma$.  Also let \bes \fg=\mathfrak{k}\oplus\fp, \ees be the Cartan decomposition. Since, $\sigma$ and $\theta$ commute we have the joint decomposition \bes \fg=\mathfrak{k}\cap\fh \bigoplus \fp\cap \fh  \bigoplus \mathfrak{k}\cap \fq \bigoplus\fp\cap\fq .\ees Let $\aq$ be the maximal semisimple abelian subspace of $\fp\cap\fq$ and $A_q=\exp{\aq}$. The dimension of $\fa_q$ is called {\em split rank} of the semisimple symmetric spaces $G/H$.  Let $\mathfrak{a}_1$ be a maximal abelian subspace of $\fq$. The dimension of $\fa_1$ is called {\em rank} of the semisimple symmetric spaces $G/H$.

In this paper we restrict ourselves to the case when the split rank of $G/H$ is one. We have $\fa_q = \R{} T$ for some $T \in \fp \cap \fq$ and parametrize the elements of $A_q$ with $a_t$ for $t \in \R{}$. Then the root system $\Sigma(\fa_q, \fg)$ is $\{\pm \alpha\}$ or $\{\pm \alpha, \pm 2\alpha\}$ and denote $\fg_\alpha$ for the corresponding root space.  
Let \[{\mathfrak{n}} = \sum_{\alpha >0}\fg_\alpha, \quad\overline{\mathfrak{n}}  = \sum_{\alpha <0}\fg_\alpha.\]
Observe that $\on = \theta ({\mathfrak{n}})$. Let $\fl$ be the centralizer of $\fa_q$ in $\fg$ and
 \[\mathfrak{l}_{kq} = \fl \cap (\mathfrak{k}\cap \mathfrak{q}).\] From \cite[Lemma 3.3]{vB87} we have the decomposition
\begin{equation}\label{eq:decompose}
    \fg = \overline{\mathfrak{n}} \oplus \mathfrak{l}_{kq} \oplus \fa_q \oplus \mathfrak{h}.
\end{equation}

 For $\Sigma(\fa_q,\fg)$, the associated Weyl group is $W = \{\pm1\}$.  Let $W_{K\cap H} = N_{K\cap H}(\fa_q)/Z_{K\cap H}(\fa_q)$ and define
\bes \mathcal W=W/W_{K\cap H}.\ees For the split rank one case $\cW$ is either $\{1\}$ or $\{\pm 1\}$.
   Let $\fg^\pm = \fg^{\pm \sigma \theta}$, that is \bes \fg_+=\mathfrak{k}\cap\fh \bigoplus \fp\cap\fq \text{ and } \fg_-=\mathfrak{k}\cap \fq \bigoplus \fp\cap \fh.\ees 
  Let $m_\alpha = \mr{dim} \fg_\alpha$ be the multiplicity of $\alpha$ and let $m_\alpha^\pm=\dim \fg_\alpha^\pm$, where $\fg_\alpha^\pm = \fg_\alpha \cap \fg^\pm$. Let $H_0\in \aq$ be a unique element such that $\alpha(H_0)=1$.  Let \bes m_1=m_{\alpha},m_1^+=m_{\alpha}^+, m_1^{-}=m_{\alpha}^{-} ;\,\, m_2=m_{2\alpha}, m_2^+=m_{2\alpha}^+, m_2^-=m_{2\alpha}^-.
\ees   
Then the following are true for split rank one cases (\cite[p. 3]{S90}): 
\begin{enumerate}
    \item $m_1^+ + m_1^->0,$
    \item If $m_2^->0$ then $m_1^+=m_1^-$,
    \item $m_1^-=m_2^-=0$ if and only if $(\mathfrak{g}, \mathfrak{h}$) is Riemannian.
    \item If $m_1^+=m_2^-=0$ then $(\mathfrak{g}, \mathfrak{h}$) is of $\mathfrak{k}_\epsilon$ type.
\end{enumerate}

The constant $\rho$ is defined as
\begin{equation}
    \rho = \frac{1}{2}(m_1 + 2m_2 )= \frac{1}{2}(m_1^+ + m_1^{-} + 2m_2^{+} + 2m_2^{-}).
\end{equation}
and the Jacobian as 
\begin{equation}\label{eq: Jacobian}
    J(t) = (\sinh t)^{m_1^+}(\cosh t)^{m_1^-}(\sinh 2t)^{m_2^+}(\cosh 2t)^{m_2^-}.
\end{equation}
for $t \in \R{}$. According to \cite[Thm. 10]{R79}, every element $g\in G$ can be decomposed as \be \label{decomposition} g=ka_th, \quad k\in K, t \in \R{},  \text{ and } h\in H,
\ee and $a_t$ is unique upto conjugation by $W_{K\cap H}$. From this we get a polar decomposition map given by 
\[K/ Z_{K\cap H}(\fa_q) \times A_q \rightarrow G/H\]
which is a surjection (see \cite[Cor. 11]{R79}). 
Consider the disjoint union \[X_+ = \bigcup_{w \in \cW}KA_q^+w H,\]

where $A_q^+$ is a positive Weyl chamber. When viewed as a subset of $G/H$, the set $X_+$ is open and dense.  Then there exists a unique $G$- invariant measure on $X$ such that for $f \in L^1(X)$
\begin{equation}\label{eq:measure}
    \int_{X} f(x)dx = \sum_{w\in \cW}\int_K \int_{0}^{\infty}f(ka_twH)J(t)dt dk,
\end{equation}
where $dk$ is the normalized Haar measure on $K$ (see \cite[Thm. 3.9, p.24]{AO05} and \cite[Thm. 8.1.1]{S84}).

Let $P$ be a $\sigma$-minimal parabolic subgroup, i.e, a parabolic subgroup satisfying $\sigma \theta P = P$ and is minimal among all other parabolic subgroups satisfying this identity.  We  have the following theorem (see \cite[Thm. 13]{R79}):

\begin{theorem}
    Let $P$ be a $\sigma$-minimal parabolic subgroup of G and $\cW$ be as defined above.  Then
    \begin{enumerate}
        \item For each $\omega \in  \cW$, the orbit $\cO_\omega = P\omega H$ is open in $G$.
        \item The orbits $\cO_\omega$ are mutually disjoint.
        \item The disjoint union $\bigcup_{\omega \in \cW}\, \cO_{\omega}$ is dense in $G$.
    \end{enumerate}
\end{theorem}

In the split rank one cases, there are two different cases: $\cW = \emptyset$, in which case $PH$ is open and dense in $X$ and $\C^\cW = \C$. The second one being $\cW = \{-1,1\}$, in which case there are two open $P$-orbits in $X$ and $\C^\cW = \C^2$. The first case is the setup assumed by Delorme in his contributions to \cite{AO05}.

\subsection{Eisenstein Integrals} 
For $\eta \in \C^{\cW}$, $\l \in \C$ such that $\Re \l > \rho$, we define
\[j(\l,\eta) ( x) = \begin{cases}
    a^{\l-\rho}\eta_w &\text{if} \; x = man\,wh \in \cO_w\\
    0 &\text{otherwise},
\end{cases} \]
where $\eta_w$ is the $w$-th coordinate of $\eta \in \C^\cW$. The left $K$-invariant Eisenstein integrals are defined as (\cite{VS97-3}): 
\[E(\l,\eta)(gH) = \int_K j(\l,\eta)(g^{-1}k)dk,\]
where $dk$ is the normalized Haar measure on $K$. Later, these Eisenstein integrals were identified as matrix coefficients of generalized principal series representations in \cite{vDS05-2}. The matrix coefficients of generalized principal series representations were initially studied by Delorme in \cite{D99}, and van den Ban and Schlitchkrull subsequently established this connection in \cite{vDS05-2}.
Let $P$ be the $\sigma$ minimal parabolic subgroup and $P = MAN$ be the Langlands decomposition. Let $\pi_{1, \l}= \mr{Ind}_{P}^G(1 \otimes \l \otimes 1)$ be the parabolically induced representations on the space of continuous functions $C(K/M)$. The representations $\pi_{1,\l}$ are unitary when $\l \in i\R{}$. The left $K$-invariant Eisenstein integrals, denoted by $E(\l)$, are linear combinations of matrix coefficients for the induced representations $\pi_{1,\l}$ with $\l \in \C$. Let $1_{-\l}$ be the $K$-fixed vector as a constant function $1$ on $K$. Observe that for each $\eta \i \C^\cW$, $j(\l,\eta) \in C^{-\infty}(K;1)^H$ is an $H$-fixed distribution of $\pi_{\l,1}$ for $\l \in \C$ such that $\Re \l > \rho$. 
\begin{theorem}
    The map $\l \mapsto j(\l,\eta)$ can be extended as a meromorphic function on $\C$ for each $\eta \in \C^\cW$. Moreover, for almost all $\l \in \C$, the map $j(\l,\cdot) : \C^{\cW} \rightarrow C^{-\infty}(K,1)^H$ is a bijection.
\end{theorem}
The proof of this theorem is given in \cite[Thm. 5.10, p.83]{vDB88} and \cite{Olafsson1987}.
The Eisenstein integrals are represented as 
\[E(\l,\eta)(x) = \langle 1_{-\l}, \pi_{\l,1}(x)j(\l,\eta)\rangle_{L^2(K)}.\]

The function $\l \mapsto E(\l,\eta)(x)$ may have singularities on $i\R{}$. Thus, we consider a normalization by their behavior at infinity which regularizes these functions $i\R{}$. We denote these normalized Eisenstein integral as $\eE(\l, \cdot)(\cdot)$. We will now discuss some properties of Eisenstein integrals. 
\newline

Let $\mathbb{D}(G/H)$ be the space of $G$ invariant differential operator on $G/H$. Then the Laplace-Beltrami operator $\Delta$ on $X$ is an element of $\mathbb{D}(G/H)$. In the rank one case, $\mathbb{D}(G/H)$ is given by the polynomial algebra $\C[\Delta]$ (cf. \cite[Lemma 2.1]{vdB92} and also the discussions in \cite[Part I,p. 25]{AO05} and \cite[Part II, p. 126, 127 and 128]{HS94}). The radial part of $\Delta$ in spherical coordinates is given by (see \cite[p. 8 and 9]{S90} and \cite[Thm. 2.1]{AM90}
\begin{equation}\label{eq:LaplaceBeltrami}
    \begin{split}
        L(\Delta) &= \frac{\partial^2}{\partial t^2} +(m_1^+ \coth t + m_1^- \tanh t + 2m_2^+\coth 2t +2m_2^- \tanh 2t) \frac{\partial}{\partial t}\\
    &= \frac{1}{J(t)}\frac{\partial}{\partial t}\left(J(t)\frac{\partial}{\partial t} \right).
    \end{split}
\end{equation}

Let $x_0 = eH$ be the base point of $G/H$. For a fixed $w \in \cW$ and $\eta \in \C^{\cW}$, the normalized Eisenstein integral satisfies the following differential equation (see \cite[Cor. 16.2]{vdB92})
\[L(\Delta) \eE(\lambda, \eta)(a_t) = (\rho^2-\lambda^2)\eE_w(\lambda,\eta)(a_t).\]

Moreover, it satisfies the asymptotic expansion (\cite[Prop. 7.7]{HS94} and \cite[Sec.16]{vdB92})
\begin{equation}\label{eq:asym}
    \lim_{t \rightarrow \infty} e^{(\rho-\lambda)t}\eE(\lambda,\eta)(a_tw) = \eta_w, \quad \Re \lambda >0,
\end{equation}
where $\eta_w$ is the $w$-th coordinate of $\eta \in \C^\cW$. We define
 \[\eE_w(\lambda,\eta)(t) := \eE_w(\lambda,\eta)(a_t w\cdot x_0).\]
 
By substituting $z = -\sinh^2 t $, the above differential equation can be transformed into a hypergeometric differential equation. Since, $E_w$ is regular at $t=0$ for all $\lambda$ we have
\[\eE_w(\lambda,\eta)(t) = \mr{const.} {}_2F_1\left(\frac{\rho+\lambda}{2}, \frac{\rho-\lambda}{2}; \frac{m_1^++m_2^++1}{2}; -\sinh ^2 t\right).\]

Using the following property (\cite[eq. 17, p. 63]{EMO81})

\begin{equation}\label{eq:hyperasymp}
\begin{split}
{}_2F_1(a, b ; c ; z) & =\frac{\Gamma(c) \Gamma(a-b)}{\Gamma(a) \Gamma(c-b)}(-z)^{-b} F\left(b, 1-c+b ; 1-a+b ; z^{-1}\right)  \\
& +\frac{\Gamma(c) \Gamma(b-a)}{\Gamma(b) \Gamma(c-a)}(-z)^{-a} F\left(a, 1-c+a ; 1-b+a ; z^{-1}\right),
\end{split}
\end{equation}

and from (\ref{eq:asym}) we obtain the explicit form for the constant (see \cite{PS25}):

\begin{equation}\label{eq:eisen}
    \eE_w(\lambda,\eta)(t) = \eta_w 2^{\lambda-\rho}\frac{\Gamma\left(\frac{\rho+\lambda}{2}\right)\Gamma\left(\frac{-\rho+\lambda+m_1^++m_2^++1}{2}\right)}{\Gamma(\lambda)\Gamma\left(\frac{m_1^++m_2^++1}{2}\right)} {}_2F_1\left(\frac{\rho+\lambda}{2}, \frac{\rho-\lambda}{2}; \frac{m_1^++m_2^++1}{2}; -\sinh ^2 t\right).
\end{equation}

Moreover, from \cite[Thm 7.6]{VS97} there exists a unique {\em endomorphism valued} meromorphic function $C^0(-1,\lambda)$ 
on $\C$ such that
\begin{equation}\label{eq: HCexpansion}
    \eE_w(\lambda, \eta)(t) = \Phi_\lambda(t)\eta_w + \Phi_{-\lambda}(t)[C^0(-1, \lambda)\eta]_w
\end{equation}
where \[\Phi_{\lambda}(t) = e^{(\lambda-\rho)t}\sum_{m=0}^\infty \Gamma_m(\lambda)e^{-mt},\]
with $\Gamma_0(\lambda) = 1$ and $\Gamma_m(\lambda)\in \C$. For $\l \notin 1/2\bN_0$, the Harish-Chandra series $\Phi_\l(t)$ also satisfies the differential equation (cf. \cite[Cor. 9.3]{VS97})
\[L(\Delta)\Phi_\l(t) = (\l^2-\rho^2)\Phi_\l(t).\]

In \cite[Lemma 2.4]{PS25} it was shown 
that the matrix \[C^0(-1,\l) = c(\l)Id\] is a diagonal matrix and 
\[\eE_w(\l,\eta)(t) = \eta_w[\Phi_\l(t) + c(\l)\Phi_{-\l}(t)],\]
where 
\begin{equation}
    c(\l) = 2^{2\lambda}\frac{\Gamma((\rho+\lambda)/2)\Gamma(-\lambda)\Gamma((\lambda-\rho+1+m_1^+ + m_2^+)/2)}{\Gamma((\rho-\lambda)/2)\Gamma(\lambda)\Gamma((-\rho - \lambda + 1+m_1^+ + m_2^+)/2)}.
\end{equation}

From (\ref{eq:eisen}) we observe that the poles and zeros of $\eE_w$ are given as follows:

\begin{proposition}
The poles of $E^\circ_w(\cdot, \eta)(t)$ are simple and equal to $\{-\rho-2k_1:k_1 \in\mathbb N_0\} \cup \{\rho - 1-m_1^+-m_2^+ -2k_2: k_2\in\mathbb N_0\}$. 
\end{proposition}
\begin{Remark}
  For $\l \in i\mathbb{R}$, the function $E^\circ_w(\lambda,\cdot)(\cdot)$ has no poles. Indeed, if $\rho - 1 - m_1^+ - m_2^+ \in 2\bN_0$, then the corresponding pole is canceled by the Gamma function in the denominator.
\end{Remark}

\begin{lemma}\label{lemma:gamma}
Fix $R>0$ and let \[ \fa^*(R)=\{\lambda\in\mathbb C\mid \Re\lambda\leq R\}.\] 
\begin{enumerate}
    \item There exists a polynomial $p_R(\l)$ depending on $R$ such that $p_R(\l)\eE(\l,\cdot)(\cdot)$ is holomorphic in the region $-\fa^*(R)$. The polynomial $p_R$ is given by \begin{equation}\label{exp:p_R}p_R(\lambda)=\prod_{\substack{k_1\in \bN_0, \\ \rho+2k_1\leq R}}(\lambda +\rho + 2k_1) \prod_{\substack{k_2\in \bN_0,\\ -\rho+ m_1^+ + m_2^+ +1+2k_1\leq R}}(\lambda -\rho +m_1^+ + m_2^+ +1 +2k_2).
    \end{equation}
    \item  Let $X_R= \{k \in \mathbb N_0: 0< k \leq R\}$. Then there exist a polynomial 
    \be \label{poly:qr}
    q_R(\lambda) = \Pi_{k \in X_R}(2\lambda-k), 
    \ee
    such that $q_R \Phi_\lambda$ is holomorphic in $\fa^*(R)$.\begin{enumerate}
        \item The poles of $\Gamma_m(\l)$ are at most simple poles and contained in $\{1/2,1,\ldots,m/2\}$.
        \item There exists constants $M, \chi >0$ depending on $R$, such that $q_R(\l)\Gamma_m(\l)$ is holomorphic in $\fa^*(R)$ and 
\[|q_R(\lambda)\Gamma_m(\lambda)| \leq M(1+m)^{\chi}(1+|\lambda|)^{\mathrm{\text{deg}} \,q_R},\]
for all $m \geq 0$ and $\lambda \in \overline{\fa^*(R)}$.
\item  Moreover, for every $\delta >0$ there exists a constant $M_\delta$ such that
\begin{equation}\label{est phi}
    |q_R(\l)\Phi_\l(t)| \leq M_\delta (1+|\l|)^{\text{deg } q_R}e^{(|\Re \l|-\rho)t}
\end{equation}

for all $\lambda \in \overline{\fa^*(R)}$ and for all $t \geq \delta$.
    \end{enumerate}
     
\end{enumerate}
\end{lemma}
The first part follows from \cite[Prop. 10.3 and Cor. 16.2]{vdB92}. The second part follows from \cite[Lemma 7.3 and Thm. 9.1]{VS97}. Also, see \cite{PS25}.

We also have the following relation:

\begin{lemma}
For $t>0$ and $\l \notin \mathbb{Z}$, we have the identity 
\begin{equation}\label{eq:eisenhyp}
    \eE_w(\l,\eta)(t) = \eta_w [\Phi_\l^0(t) +c(\l)\Phi_{-\l}^0(t)] ,
\end{equation}
where \[\Phi_\l^0(t) = (2\cosh t)^{\l-\rho}{}_2F_1\left(\frac{\rho-\l}{2}, \frac{-\rho-\l+1+m_1^++m_2^+}{2}; 1-\l; \cosh^{-2}t \right).\]
\end{lemma}
\begin{proof}
    The relation (\ref{eq:eisenhyp}) is obtained from \cite{EMO81}[(17) p.63, (22) p.64]. 
\end{proof}

We also have that $\Phi_\l^0(t)$ is also an eigenfunction of the radial part $L(\Delta)$ with eigenvalue $(\l^2-\rho^2)$.

\begin{proposition}\label{prop:serieshyp}
    For $\l \notin 1/2\bN_0$ we have
    \[\Phi_\l^0(t) = \Phi_\l(t)\]
\end{proposition}
\begin{proof}
    We know that both satisfy the differential equation $L(\Delta) \Phi = (\l^2-\rho^2)\Phi$, which can be transformed into a second order hypergeometric differential equation. Moreover, both satisfy the asymptotic behavior
    \[\lim_{t \rightarrow \infty} e^{(\rho-\l)t}\Phi(t) = 1.\]
    Thus, we have $\Phi_\l^0(t) = \Phi_\l(t)$ for $t>0$ and $\l \notin 1/2\mathbb{N}_0$ as $\Phi_\l$ have singularities at these points. 
\end{proof}

Let $\alpha, \beta, \l \in \C$ with $\alpha \neq -1,-2,\ldots$ and $\rho = \alpha+\beta +1$. From \cite[Eq. 2.3]{K75} the Jacobi functions with paramaters, $\alpha, \beta$ and $\l$ is defined as
\[\varphi_{-i\l}^{(\alpha,\beta)}(t) = {}_2F_1\left(\frac{\rho+\l}{2},\frac{\rho-\l}{2};\alpha +1; -\sinh^2t\right).\]

Then the normalized Eisenstein integral can also be written in the form of Jacobi functions for $\Re \l >0$ with parameters $\alpha = (m_1^+ + m_2^+ +1)/2 -1$, $\beta = (m_1^- + m_2^+ + 2m_2^- +1)/2-1$ and $\rho = \alpha+\beta+1$  as 
\[\eE_w(\l,\eta)(t) = \eta_w2^{\lambda-\rho}\frac{\Gamma(\frac{\lambda+\rho}{2})\Gamma(\frac{\lambda-\rho+m_1^+ + m_2^+ +1}{2})}{\Gamma(\lambda)\Gamma(\frac{m_1^+ + m_2^+ +1}{2})} \varphi_{-i\l}^{(\alpha,\beta)}(t).\]

\begin{lemma}\label{lemma:est}
    Let $n,m \in \bN_0$ and fix $R>0$. Then there exists a constant $A,A' >0$ such that for $\l \in - \fa^*(R)$ the following estimates hold:
    \begin{equation}
    \begin{split}
         \left|\frac{\partial^m}{\partial t^m} p_R(\l)\eE_w(\l,\eta)(t) \right| &\leq A \|\eta\| (1+t)(1+|\l|)^{\text{deg} \,p + m}e^{(|\Re \l|-\rho)t};\\
         \left|\frac{\partial^n}{\partial \l^n} p_R(\l)\eE_w(\l,\eta)(t) \right| & \leq A'\|\eta\| (1+t)^{n+1}(1+|\l|)^{\text{deg} \,p }e^{(|\Re \l|-\rho)t}.
    \end{split}
       \end{equation}
\end{lemma}
\begin{proof} 
   From \cite[Lemma 2.3]{K75} the Jacobi function $\varphi^{\alpha,\beta}_{-i\l}$  can be estimated as 
   \[\left|\frac{\partial^m}{\partial t^m}\varphi^{\alpha,\beta}_{-i\l}(t)\right| \leq A (1+|\l|)^m(1+t)e^{(|\Re \l|-\rho)t}.\] Using \cite[1.18 (3)]{EMO81} we can estimate 
   \[p_R(\l)2^{\lambda-\rho}\frac{\Gamma(\frac{\lambda+\rho}{2})\Gamma(\frac{\lambda-\rho+m_1^+ + m_2^+ +1}{2})}{\Gamma(\lambda)\Gamma(\frac{m_1^+ + m_2^+ +1}{2})}\] with $(1+|\l|)^{\text{deg}\; p_R}$ (see \cite{PS25}). 
   Proof of the second estimate follows from first estimate with $m=0$ and the Cauchy's integral formula by integrating over a circle with radius $\frac{1}{1+t}$. 
\end{proof}

\subsection{Fourier transform and Fourier inversion:} 
For a left $K$-invariant function $f \in C_c^\infty(X)^K$, the Fourier transform of $f$ is defined by 
\begin{align*}
     \widehat{f}(\l)(\eta) &= \int_X f(x)\eE(-\l,\eta)(x)dx\\
    &=\sum_{w \in \cW} \int_{0}^\infty f(a_t w \cdot e_1)E^0(-\lambda, \eta)(a_tw \cdot e_1)J(t)dt,
\end{align*}
for all $\eta\in \C^{\cW}$, $\l \in i\R{}$ and $J(t)$ is Jacobian given in equation (\ref{eq: Jacobian}).
Moreover, \[\widehat{f}(-\lambda)(\eta) = \widehat{f}(\lambda)(C^0(-1,\lambda)\eta).\]

On $i\R{}$, the function $\eE_w$ has no pole and is a bounded function in $t$ (Lemma~\ref{lemma:est}). Thus, the definition of Fourier transform can be extended to any left K-invariant $L^2$ function on $X$.

Since, $\widehat{ f}(\l)$ is a linear form on $\C^\cW$, we define $\cF {f}(\l)\in \C^\cW$ such that
\[\langle\cF (f)(\l), \eta \rangle_{\C^\cW} := \langle f, \eE(-\overline{\l},\eta) \rangle_{L^2(X)} = \widehat{ f}(\l)\eta, \qquad \text{for all}\; \eta \in \C^\cW,\]
where $\langle\cdot,\cdot \rangle_{\C^\cW}$ is the standard inner product on $\C^\cW$ and $\langle\cdot,\cdot \rangle_{L^2(X)}$ is the $L^2$ inner product on $X$. 

Now, for a $\C^\cW$- valued function $\phi$ on $i\R{}$ we define the wave packet $\cJ$ as 
\[\cJ (\phi)(x) = \int_{i\R{}}\eE(\l,\phi(\l))(x)d\l\qquad x \in X.\]

Let \[L = \{\l_k = \rho -1-m_1^+-m_2^+ -2k: k\in \bN_0\; \text{and}\; \l_k > 0\},\]
which is a finite set. Notice that
$L= \emptyset$ in the Riemannian case. Let $f \in C_c^\infty(X)^K$. The Fourier Inversion formula for the general rank cases is discussed in \cite[Thm. 7.1]{vS99} and the explicit formula given below for the split rank one symmetric spaces is given in \cite[Eq. 5-13, p.128]{AO05}. 
    \begin{align*}
        f(a_tw\cdot x_0) &= \cJ (\cF {f})(a_t) + 4\pi i\sum_{\l_k \in L} \mr{Res}_{\l = -\l_k} [\Phi_\l (t) \cF{f} (\l)]_w\\
        &=\int_{i\R{}} \eE_w(\lambda, \cF{ f}(\lambda))(t) d\lambda + 4\pi i\sum_{\l_k \in L} \mr{Res}_{\l = -\l_k} [\Phi_\l (t) \cF{f} (\l)]_w.
    \end{align*}

Note that $\Phi_\l = \Phi_\l^0$ for $\l \notin 1/2\mathbb{Z}$. Hence, \bes \mr{Res}_{\l = -\l_k}[\Phi_\l(t) \cF f(\l)]_w = \mr{Res}_{\l = -\l_k}[\Phi_\l^0(t) \cF f(\l)]_w = \Phi_{-\l_k}^0(t) \mr{Res}_{\l = -\l_k}[\cF f(\l)]_w.\ees Using this and Proposition~\ref{prop:serieshyp} we can rewrite the Fourier inversion formula as follows:

\begin{theorem}[Fourier Inversion]\label{thm:Finversion}

    For $f \in C_c^\infty(X)^K$ the Fourier inversion formula is given by
    \begin{equation}\label{inversion-1}
        f(a_tw\cdot x_0) = \int_{i\R{}} \eE_w(\lambda, \cF{ f}(\lambda))(t) d\lambda + 4\pi i\sum_{\l_k \in L} \Phi_{-\l_k}^0 (t) \mr{Res}_{\l = -\l_k}[\cF{f} (\l)]_w.
    \end{equation}
 
\end{theorem}

\begin{Remark}
  Let the discrete part of the space $L^2(X)^K$, denoted by $L^2(X)^K_d$ and let $P: L^2(X)^K \rightarrow L^2(X)^K_d$ be the orthogonal projection. Let
\[Tf(a_tw\cdot x_0) :=  4\pi i\sum_{\l_k \in L} \mr{Res}_{\l = -\l_k} [\Phi_\l (t) \cF{f} (\l)]_w.\]
It is seen in \cite[Thm. 21.2, Def. 12.1 and Lemma 12.6]{vdS05} that $T$ is the restriction to $C_c^\infty(X)^K$ of the orthogonal projection $P$. 
Let $\mathcal{C}^2(X)^K$ be the space of $L^2$-Schwartz functions on $X$ defined in Section 3 and let $\mathcal{A}_2(X)^K$ be the space of $\mathbb{D}(X)$-finite functions in $\mathcal{C}^2(X)^K$. It was shown in \cite[Lemma 12.6]{vdS05} that $\mathcal{A}_2(X)^K = L^2(X)^K_d$. Thus, for $f \in C_c^\infty(X)^K$ the function $Tf$ lies in the discrete part $L^2(X)^K_d = \mathcal{A}_2(X)^K$.
\end{Remark}

\section{$L^r$ - Schwartz space}

Let $f$ be a K-invariant smooth function $X$. Then using the $KA_qH$ decomposition and parameterizing $A_q$ with $\R{}$,  we can write $f(t) := f(ka_th)$ for $t \in \R{}$.
For $0 < r \leq 2$, we define the $L^r$-Schwartz space $\mathcal{C}^r(X)^K$ on $X$ as the space of all left K-invariant smooth functions on $X$ such that for all $n \in \bN_0$ and for any $D \in \mathbb{D}(X)$,  we have
\[\tau_{n, D}(f) = \sup_{t >0} (1+t)^ne^{(2/r)\rho t}\left|D f(a_t\cdot x_0)\right| < \infty.\]
We multiply by the factor $e^{(2/r)\rho t}$ in the definition of $\tau_{n,m}$ so that the function $f$ is in $L^r(X)$. We now claim that in this split rank one case, $f\in\mathcal C^r(X)^K$ if and only if for all $m, n\in \mathbb N_0$, \be \label{defn:schwartz}
\sup_{t>0}(1+t)^n e^{(2/r)\rho t}\left|\left(\frac{d}{dt}\right)^m f(a_t\cdot x_0)\right| < \infty. 
\ee
To prove this claim we first consider the functions \[r(t) = e^{-t}(e^t-e^{-t})^{-1} \quad\text{and} \quad s(t) = (e^t-e^{-t})^{-1}.\] Let $\mathcal{A}$ be the algebra generated by the functions $\{1, r(t), s(t), r'(t),s'(t),\ldots\}$. Denote $L_u$ and $R_u$ as the infinitesimal action of $u \in \mathcal{U}(G/H)$ on $C^\infty(G)$, induced by the left and right regular representation $L$ and $R$, respectively. Let $D \in \mathbb{D}(G/H)$, then there exists an ${u} \in \mathcal{U}(G/H)^H$ such that $D = R_u$. 

By \cite[Lemma 3.5]{vB87}, for every element $X_{-\alpha} = \theta (X_\alpha) \in \overline{\mathfrak{n}}$, for $f \in C^\infty(X)^K$ and for any $t>0$, we can write
\begin{equation}\label{eq:nbardecompose}R_{X_{-\alpha}}f(a_t\cdot x_0)
=
s(t)\,R_{\operatorname{Ad}(a_t^{-1})(X_\alpha+\theta X_\alpha)}
f(a_t\cdot x_0)
-
r(t)\,R_{X_\alpha+\tau X_\alpha}
f(a_t\cdot x_0),
\end{equation}
 where $X_\alpha +\theta X_\alpha \in \mathfrak{k}$ and $X_\alpha + \tau X_\alpha \in \mathfrak{h}$. In order to prove (\ref{defn:schwartz}), we now prove the following estimate for the invariant differential operator.

\begin{lemma}
    For any $D \in \mathbb{D}(G/H)$ and for a fixed $\delta >0$, there exists a constant $C_\delta$ and a positive integer $n \in \bN$ such that
    \begin{equation*}
        |Df (a_t)| \leq C_\delta\sum_{i=1}^n \left|\frac{d^i}{d t^i} f(a_t)\right|,
    \end{equation*}
    for $t \geq \delta$ and $f \in C^\infty(X)^K$. 
\end{lemma}
\begin{proof}
    We will use induction on the degree of $u$. Using (\ref{eq:decompose}) any $X \in \fg$ can be decomposed as
    \[X = X_{\overline{\fn}} + X_{\mathfrak{l}_{kq}} + c_XT + X_{\fh}.\]
Consequently, $R_X = R_{X_{\overline{\fn}}} + R_{X_{\mathfrak{l}_{kq}}} + R_{c_X T} + R_{X_{\fh}}$. Let $f \in C^\infty(X)^K$. Since, $f$ is a right $H$-invariant function, we obtain
\begin{equation*}
    R_{X_{\fh}}f(a_t) = \frac{d}{d s}\Big|_{s =0} f(a_t \mr{exp} (sX_{\fh}))= \frac{d}{d s}\Big|_{s =0} f(a_t)=0.
\end{equation*}
Next, we will see that the $c_XT$ part gives a first order derivative.
\begin{align*}
    R_{c_X T} f(a_t) &= \frac{d}{d s}\Big|_{s =0} f(a_t \mr{exp} (s c_XT)) = \frac{d}{d s}\Big|_{s =0} f(\mr{exp}(t T) \mr{exp} (s c_XT))\\
    &= \frac{d}{d s}\Big|_{s =0} f(\mr{exp}((t +s c_X)T)) = c_X \frac{d}{d t} f(a_t).
\end{align*}
    Similar to $X_\fh$ part we will get that $X_{\fl_{kq}}$ also gives zero. This is because,  $\mr{exp}(sX_{\fl_{kq}})$  lies in the centralizer of $A_q$ in $K$ and $f$ is a left $K$-invariant function. Thus,
    \begin{align*}
        R_{X_{\fl_{kq}}} f(a_t) &= \frac{d}{d s}\Big|_{s =0} f(a_t \mr{exp} (s X_{\fl_{kq}})) = \frac{d}{d s}\Big|_{s =0} f( \mr{exp} (s X_{\fl_{kq}}) a_t )\\
        & = \frac{d}{d s}\Big|_{s =0} f(a_t ) = 0.
    \end{align*}
    Without loss of generality, assume that $X_{\on} = \theta X_{\alpha}$ for some $\alpha >0$. From (\ref{eq:nbardecompose}) we have 
    \[X_{\on} = (e^t-e^{-t})^{-1}Ad(a_t^{-1}) (X_\alpha +\theta X_\alpha)  -e^{-t}(e^t-e^{-t})^{-1} (X_\alpha + \tau X_\alpha).\]
Then, $R_{X_{\on}} = (e^t-e^{-t})^{-1} R_{Ad(a_t^{-1}) (X_\alpha +\theta X_\alpha)} -e^{-t}(e^t-e^{-t})^{-1} R_{(X_\alpha + \tau X_\alpha)}$. As, $X_\alpha + \tau X_\alpha \in \fh$, we obtain that $R_{(X_\alpha + \tau X_\alpha)} f(a_t) =0$. Furthermore,
\begin{align*}
    R_{Ad(a_t^{-1}) (X_\alpha +\theta X_\alpha)} f(a_t) &= \frac{d}{d s}\Big|_{s =0} f(a_t \mr{exp} (sAd(a_t^{-1}) (X_\alpha +\theta X_\alpha) )) = \frac{d}{d s}\Big|_{s =0} f(a_t Ad(a_t^{-1}) \mr{exp} (s (X_\alpha +\theta X_\alpha) ))\\
    &= \frac{d}{d s}\Big|_{s =0} f( \mr{exp} (s (X_\alpha +\theta X_\alpha)  )a_t) = \frac{d}{d s}\Big|_{s =0} f( a_t ) =0,
\end{align*}
as $X_\alpha +\theta X_\alpha \in \fk$. Therefore, we obtain $R_{X_\fn} f(a_t) = 0$ and 
\[R_X f(a_t) = c_X \frac{d}{d t} f(a_t).\]

Assume that for $u \in \mathcal{U}(G/H)_{n-1}$, we have
\[R_u f(a_t) = \sum_{i=0}^{n-1}z_i(t) \frac{d^i}{dt^i} f(a_t)\]
with $z_i(t) \in \mathcal{A}$. Let $X Y \in \mathcal{U}(G/H)_{n}$ with $X \in \fg$ and $Y \in \mathcal{U}(G/H)_{n-1}$. Let $X = {X_{\overline{\fn}} + X_{\mathfrak{l}_{kq}} + c_XT + X_{\fh}}$ and without loss of generality assume that $X_{\on} = \theta X_\alpha$. Then
\begin{equation*}
    R_{XY} f(a_t) = R_X R_Y f(a_t) = \Big(s(t) R_{Ad(a_t^{-1}) (X_\alpha +\theta X_\alpha)} - r(t) R_{(X_\alpha + \tau X_\alpha)}+ R_{X_{\mathfrak{l}_{kq}}} + R_{c_XT} + R_{X_{\fh}}\Big)(R_Y f)(a_t).
 \end{equation*}
Since, left and right derivatives commute we see that $R_{Ad(a_t^{-1}) (X_\alpha +\theta X_\alpha)} (R_Y f)(a_t) = L_{X_\alpha +\theta X_\alpha}R_Y f(a_t) = R_Y L_{X_\alpha +\theta X_\alpha} f(a_t) =0$. Because $X_{\fl_{kq}} \in \fk$ centralizes $\fa_q$, we also obtain $R_{X_{\fl_{kq}}}R_Y f(a_t) = L_{X_{\fl_{kq}}} R_Y f(a_t) = R_Y L_{X_{\fl_{kq}}} f(a_t) =0$. 
Now, 
\[\begin{split}
    R_{c_XT} (R_Y f)(a_t) &= c_X \frac{d}{d t} R_Yf(a_t) = c_X \frac{d}{d t} \left(\sum_{i=0}^{n-1}z_i(t) \frac{d^i}{dt^i} f(a_t)\right)\\
    &= c_Xz_0'(t) f(a_t) + c_X\sum_{i=1}^{n-1} [z_{i-1}(t)+z_i'(t)] \frac{d^i}{dt^i} f(a_t) + c_Xz_{n-1}(t) \frac{d^n}{dt^n} f(a_t) .
\end{split}\]

Now we consider the $\fh$-part. We have
\begin{align*}
    r(t) R_{(X_\alpha + \tau X_\alpha)} R_Y f(a_t) &= r(t)\left(R_Y R_{(X_\alpha + \tau X_\alpha)} f(a_t) + R_{[(X_\alpha + \tau X_\alpha), Y]} f(a_t)\right) \\&=r(t) R_{[(X_\alpha + \tau X_\alpha), Y]} f(a_t) \\
    & = r(t)\sum_{i=0}^{n-1}z_i(t) \frac{d^i}{dt^i} f(a_t), \quad \text{for some $z_i \in \mathcal{A}$},
\end{align*}
as $R_{(X_\alpha + \tau X_\alpha)} f(a_t) = 0$ and $[X_\alpha + \tau X_\alpha, Y] \in \mathcal{U}(G/H)_{n-1}$. We also obtain the same for $X_\fh$ part following the same steps.
Thus,
\[R_X R_Y f(a_t) = \sum_{i=0}^{n}z_i(t) \frac{d^i}{dt^i} f(a_t), \quad \text{for some $z_i \in \mathcal{A}$}.\]
Now, fix $\delta >0$. Then for all $i = 0,1,\ldots,n$, the functions $z_i \in \mathcal{A}$ are bounded for $t \geq \delta$. Therefore, we obtain the desired estimate. 
\end{proof}

\begin{Remark}
As a result of this lemma, it suffices to consider the derivatives with respect to the radial variable $t$. Note that the Laplace-Beltrami operator $\Delta$ acts radially on left $K$-invariant functions. By (\ref{eq:LaplaceBeltrami}), the radial part $L(\Delta)$ is a polynomial in $\frac{\partial}{\partial t}$. Hence the claim (\ref{defn:schwartz}) is established. 
\end{Remark}

Then $\mathcal{C}^r(X)^K$ is a Fr\'echet space equipped with the seminorm $\tau_{n, D}$. We observe that  $\mathcal{C}^r(X)^K \subseteq \mathcal{C}^{r'}(X)^K$ for $r\leq r'$.
 It is also easy to check that \bes \mathcal{C}^r(X)^K \subseteq L^r(X)^K, \ees where $L^r(X)^K$ is the space of functions in $L^r(X)$ which are left $K$-invariant.
 
We notice from (\ref{eq:eisen}) that $\eE_w$ has no poles in the region $-1/2 < \Re \l < 0$ and hence $\cF(\l)$ has no pole in $0<\Re\lambda<1/2$. Thus, we fix an $0 < \epsilon_0 < 1/2$. 
For each $ 0 < r\leq 2$, we recall \bes \gamma_r = (2/r-1)\rho, \text{ and } {S_r := i\R{} + [-\gamma_r, \epsilon_0]}.\ees  
For $k\in\mathbb N_0$, we let \bes \l_k = \rho -1-m_1^+-m_2^+ -2k.\ees We note that $\lambda_k$'s are poles of $\eE_w(\l,\eta)(t)$ and hence $-\lambda_k$'s are poles of $\cF f(\l)$.
For each $0 <r\leq 2$, let $p_r$ be a fixed polynomial which has zeros only at $-\l_k$ where $|\l_k|<\gamma_r$ in $S_r$. We fix the polynomial $p_r$ through rest of the article.

We define the Schwartz space $\mathscr{S}(S_r)_e$ consists of all functions $\phi: S_r\rightarrow\C^\cW$ such that
\begin{enumerate}
    \item $\phi$ is smooth on $i\R{}$.
    \item For $r\neq 2$, the function $p_r(\l)\phi(\l)$ and all its derivatives are holomorphic on $S_r^\circ$ and extend continuously to $S_r$.
    \item $\phi$ has simple poles at $-\l_k$ with $0<\l_k < \gamma_r, k\in\mathbb N_0$.
    \item For any polynomial $q$ and $n\in \mathbb N_0$
    \[\omega_{n,q}(\phi) = \sup_{\l \in S_r}(1+|\l|)^{n-\text{deg}\; p_r}\, \left\|q\left(\frac{\partial}{\partial\l}\right)p_r(\l)\phi(\l)\right\|<\infty.\]
    \item $\phi(-\l) = C^0(-1,\l)\phi(\l)$.
    \end{enumerate}
In the definition of $\omega_{n,q}$ we divide by the factor $(1+|\l|)^{\text{deg}\; p_r}$ to compensate for the polynomial growth introduced by multiplying $\phi$ with $p_r$. The topology on $\mathscr{S}(S_r)_e$ defined by the seminorms $\omega_{n,q}$ makes it a Fr\'echet space. Moreover, 
\[\mathscr{S}(S_r)_e \subset \mathscr{S}(i\R{})_e,\]
for $0 < r < 2$ (in the sense that, if $f\in \mathscr{S}(S_r)_e$, then $f|_{i\mathbb R}\in \mathscr{S}(i\R{})_e$). We will now collect some facts about $L^2$ Schwartz space in the following theorem.
\begin{theorem}\label{thm:S2}
\begin{enumerate}
    \item  The Fourier transform $\cF: \mathcal{C}^2(X)^K \longrightarrow \sS(i\R{})_e$ is a continuous linear  map.
    \item The map $\cJ: \sS(i\R{})_e \longrightarrow \mathcal{C}^2(X)^K$ is a continuous linear  map. 
    \item Moreover, the operator $\cF \cJ$ is a continuous identity operator on $\sS(i\R{})_e$.
    \item The kernel of $\cF$ is given by $\sS^2_d(X)^K := ker (\cJ \cF)$ and the image of $\cJ$ is given by $\sS^2_c(X)^K:= ker (\cJ - \cJ \cF)$. 
\end{enumerate}
\end{theorem}
\begin{proof}
Part (1) is proved in \cite[Lemma 16.2]{VS97-3} and part (2) is proved in \cite[Thm. 16.4]{VS97-3}.
    From \cite[Thm 16.6]{VS97-3} we have part (3). The proof of part (4) is given in \cite[Prop. 17.3]{VS97-3}. The following is used in the proof 
\begin{equation}\label{eq:FIF=F}
    \cF \mathcal{J} \cF = \cF \quad \text{on} \; \mathscr{S}^2(X)^K,
\end{equation}

 which follows from part (3).
\end{proof}
The Fourier inversion formula also holds for functions in $L^2$-Schwartz space which was proven in \cite[Thm 21.2]{vdS05} and independently in \cite[Thm. 2]{D99}. We recall that \[L = \{\l_k = \rho -1-m_1^+-m_2^+ -2k: k\in \bN_0\; \text{and}\; \l_k > 0\}.\]

\begin{theorem}\label{thm:FISchwartz}
    Let $f \in \mathcal{C}^2(X)^K$. Then the following inversion formula holds:
    \[f(a_tw\cdot x_0) =  \int_{i\R{}} \eE_w(\lambda, \cF{ f}(\lambda))(t) d\lambda + 4\pi i\sum_{\l_k \in L} \Phi_{-\l_k}^0 (t) \mr{Res}_{\l = -\l_k}[\cF{f} (\l)]_w.\]
\end{theorem}

Assume that $L \neq \emptyset$. Consider the function $\Phi^0_{-\l_k}(t)$ with $\l_k \in L$ defined in Lemma~\ref{eq:eisenhyp} where for $x = ka_th \in X$
\be \label{eqn:phi^0}\Phi^0_{-\l_k}(ka_th) := \Phi^0_{-\l_k}(t):= (2\cosh t)^{-\l_k-\rho}{}_2F_1\left(\frac{\rho+\l_k}{2}, -k; 1+\l_k; \cosh^{-2}t \right).\ee
We observe that ${}_2F_1\left(\frac{\rho+\l_k}{2}, -k; 1+\l_k; \cosh^{-2}t \right)$ is a polynomial in $(\cosh t)^{-2}$.
We recall that for $0 < r\leq 2$, \bes L_r = \{\l \in L: \l  < \gamma_r\} \text{ and } L_r^c=\{\l\in L: \l>\gamma_r\},
\ees
where $\gamma_r=(\frac{2}{r}-1)\rho$. Then for $0<r<1$ we have
\[L_r = L \text{ and } L_r^c=\emptyset.\]
\begin{lemma}\label{lem:Phi^0}
    For $0 < r\leq 2$, the function $\Phi^0_{-\l_k}$ is in $\mathcal{C}^r(X)^K$ if and only if $\l_k\in L_r^c$.
\end{lemma}
\begin{proof}
    It is easy to check that the function $(\cosh t)^{-\delta-\rho}$ is in $\mathcal{C}^r(X)$ if and only if $\delta>\gamma_r$. Then from (\ref{eqn:phi^0}), the result follows.
\end{proof}
Below we give the description of the kernel of $\cF: \mathcal{C}^2(X)^K \longrightarrow \sS(i\R{})_e$  explicitly.

\begin{theorem} Let $\cF: \mathcal{C}^2(X)^K \longrightarrow \sS(i\R{})_e$ be the Fourier transform. 
\begin{enumerate}
    \item If $L$ is an empty set, then the Fourier transform $\cF$ is injective.
    \item Suppose that $L$ is non empty set. Then the kernel of $\cF$ is the span of the functions $\{\Phi^0_{-\l_k}(t): \l_k \in L\}$.
\end{enumerate}
     
\end{theorem}
\begin{proof}
Suppose that $L = \emptyset$ and $f$ is an $L^2$- Schwartz function with $f \in \mr{Ker}(\cF)$.  Then by the Fourier inversion from Theorem~\ref{thm:FISchwartz} we get 
   \[f(a_tw\cdot x_0) = \int_{i\R{}}\eE_w(\lambda, \cF{ f}(\lambda))(t) d\lambda = 0.\]
   Thus, $f =0$ and hence the Fourier transform $\cF $ is an injective map.  
   
Suppose that $L \neq \emptyset$. From Lemma \ref{lem:Phi^0}, we have $\Phi^0_{-\l_k}$ is in $\mathcal{C}^2(X)^K$. 
Using the fact that the Laplace-Beltrami operator $\Delta$ is a self adjoint operator we get 
    \[\int_{0}^\infty (\Delta \Phi^0_{-\l_k}(t)) \eE_w(-\l,\eta)(t)J(t) dt = \int_{0}^\infty  \Phi^0_{-\l_k}(t) (\Delta \eE_w(\l,\eta)(t))J(t) dt, \quad \l \in i\R{}\]
    Since, $\Phi^0_{-\l_k}(t)$ and $\eE_w(-\l,\eta)(t)$ are eigenfunctions of $\Delta$ with eigenvalues $\l_k^2-\rho^2$ and $\l^2-\rho^2$, respectively. Therefore, we obtain
    \[(\l_k^2-\rho^2)\int_0^\infty\Phi^0_{-\l_k}(t) \eE_w(-\l,\eta)(t)J(t) dt = (\l^2-\rho^2)\int_0^\infty\Phi^0_{-\l_k}(t) \eE_w(-\l,\eta)(t)J(t) dt.\]
    As $\l \in i\R{}$, we have $\l^2 \neq \l_k^2$ and conclude that
    \be \label{ft:discrete}\int_0^\infty\Phi^0_{-\l_k}(t) \eE_w(-\l,\eta)(t)J(t) dt = 0.\ee
    This shows that $\Phi_{-\l_k}^0$ is in the kernel of $\cF$.
    
    Conversely, let $f$ is in kernel of $\cF$. That is, $\cF f(\l)=0$ for all $\l\in i\mathbb R$. Then by the inversion formula (\ref{inversion-1}) we have \begin{equation}\nonumber
        f(a_tw\cdot x_0) = 4\pi i\sum_{\l_k \in L} \Phi_{-\l_k}^0 (t) \mr{Res}_{\l = -\l_k}[\cF{f} (\l)]_w.
    \end{equation}
    Hence the kernel of $\cF$ is in the span of $\{\Phi^0_{-\l_k}(t): \l_k \in L\}$. 
\end{proof}

\begin{Remark}
   (1) We have $m_1^- = m_2^- = 0$ if and only if $X$ is Riemannian symmetric space. This implies that $L = \emptyset$ if $X$ is Riemannian symmetric space. Thus, the Fourier transform is injective.
   Moreover, it was proven in \cite{HC75,HC76a,HC76b} that $\cF$ is an isomorphism on $\mathcal{C}^2(X)^K$ for Riemannian spaces (also see \cite{He84}).\\
   (2) Consider the Pseudo Riemannian hyperbolic spaces defined as
   \[SO_e(p,q)/SO_e(p-1,q) \simeq \{x \in \R{p+q}: x_1^2+\ldots +x_p^2-x_{p+1}^2-\ldots -x_{p+q}^2=1\}.\]
   In the case $\rho = (p+q-2)/2$ and $L = \{\rho-q-2k: k \in \bN_0 \text{ and } \rho-q-2k>0\}$. If $p \leq q+2$ then $L$ is an empty set (see \cite[P. 77]{A01}) . Thus, for $p \leq q+2$ the Fourier transform $\cF$ is an injective map.
\end{Remark}

We now prove our main theorem Theorem \ref{thm:L-r-schwartz} for $L^r$-Schwartz spaces ($0 < r\leq 2$). We recall that, for a $\C^\cW$- valued function $\phi$ on $i\R{}$ the wave packet $\cJ$ is defined as 
\[\cJ (\phi)(x) = \int_{i\R{}}\eE(\l,\phi(\l))(x)d\l\qquad x \in X.\]

For $\phi \in \sS(S_r)_e$, using the Weyl invariance of measure $d\l$ the wave packet $\cJ$ becomes
  \begin{align*}
      \cJ \phi (a_tw\cdot x_0) &= \int_{i\R{}} \eE(\l,\phi(\l))(a_tw) d\l =\int_{i\R{}} (\Phi_\l(t) + c(\l)\Phi_{-\l}(t))[\phi(\l)]_w d\l \\
      & = 2 \int_{i\R{}} \Phi_\l(t)[\phi(\l)]_w d\l.
  \end{align*}
For $\phi \in \sS(S_r)_e$ we define,  \[\cI_r\phi (a_tw\cdot x_0):= {2}\int_{i\R{} - \gamma_r} \Phi_\l(t)[\phi(\l)]_w d\l = 2\int_{i\R{}} \Phi_{\l-\gamma_r}(t)[\phi(\l-\gamma_r)]_w d\l. \]

For the proof of Theorem~\ref{thm:L-r-schwartz}, we decompose the argument into the following lemmas, which together establish the result.

\begin{lemma}
    Let $0 < r \leq 2$ . The Fourier transform map $\cF: \mathscr{S}^r(X)^K \longrightarrow \mathscr{S}(S_r)_e$ is continuous linear  map.
\end{lemma}
\begin{proof}
    By the properties of Fourier transform we know that $\cF f(\l)$ is smooth on $i\R{}$ and $\cF f(-\l) = C(-\l)\cF f(\l) $.  Moreover, for any $f \in \mathcal{C}^r(X)^K$ and $\l \in S_r$ we have from Lemma~\ref{lemma:est}  (also see \cite{PS25})
    \begin{align*}
        \left\|\frac{\partial^m}{\partial \l^m} p_r(\l)\cF{f}(\l)(\eta)\right\| &\leq \left\| \sum_{w \in \cW}\int_0^\infty f(a_tw) \left(\frac{\partial^m}{\partial \l^m}p_r(\l)\eE_w(-\l,\eta)(t)\right)J(t)dt\right\|\\
        &\leq A \sum_{w \in \cW}\int_0^\infty|f(a_tw)| \|\eta\| (1+t)^{m+1}(1+|\l|)^{\text{deg} \;p_r}e^{(\Re \l -\rho)t} J(t)dt\\
        &\leq 2A \|\eta\|\int_0^\infty (1+t)^{-n}e^{-2/r\rho t} (1+t)^{m+1} (1+|\l|)^{\text{deg}\;p_r} e^{(2/r -2)\rho t}e^{2\rho t}dt\\
        &\leq 2A \|\eta\|(1+|\l|)^{\text{deg}\;p_r}\int_0^\infty (1+t)^{-n}(1+t)^{m+1}dt \\
        &\leq A' \|\eta\|(1+|\l|)^{\text{deg}\;p_r}
    \end{align*}
    for large enough $n \in \bN_0$ and any $\eta \in \C^\cW$. Thus, $p_r(\l)\cF(\l)$ and all its derivatives are holomorphic on the strip $S_r^0$ and extend continuously to $S_r$.
    It remains to show that it is continuous. Without loss of generality let $q(\l) = 1$. Observe that $\cF (\Delta + \rho^2) f(\l) = \l^2 \cF f(\l)$. Using this and the Binomial expansion we have 
    \begin{align*}
        \omega_{k,1}(\cF f) &= \sup_{\l \in S_r}\left\|(1+|\l|)^{k-\text{deg}\; p_r}p_r(\l)\cF f(\l)\right\|\\
        & \leq C  \sup_{\l \in S_r}(1+|\l|)^{-\text{deg}\; p_r}\sum_{i = 0}^{k}\| p_r(\l)\cF (\Delta + \rho^2)^if (\l)\|\\
        &\leq C \sup_{\l \in S_r} \int_{0}^\infty (1+|\l|)^{-\text{deg}\; p_r}\sum_{i=0}^{k} |(\Delta+\rho^2)^{i} f(t)| (1+|\l|)^{\text{deg}\; p_r} (1+t)e^{(|\Re \l| -\rho)t} e^{2\rho t} dt <\infty.
    \end{align*}
    
We will use induction on $m$ and assume that $\omega_{k,m-1}(\cF) < \infty$ for any positive integer $k$. Let $q(\l) = \l^m$. Then using the Leibniz differentiation rule, we have $\omega_{k,m}(\cF f) =$
  \begin{align*}
        & \sup_{\l \in S_r}  \|\frac{\partial^m}{\partial\l^m}[(1+|\l|^2)^{k-\text{deg}\; p}p_R(\l)\cF f(\l)] - \sum_{i=0}^{m-1} \frac{\partial^{m-i}}{\partial\l^{m-i}}(1+|\l|^2)^{k-\text{deg}\; p} \frac{\partial^i}{\partial \l^i}(p_R(\l)\cF f(\l)) \|\\
       &\leq \sup_{\l \in S_r} \|\frac{\partial^m}{\partial\l^m}(1+|\l|^2)^{-\text{deg}\;p} \sum_{i=0}^{2k}p_R(\l)\cF(\Delta+\rho)^if(t) \| + A'\sum_{i=0}^{m-1} \omega_{k-m-i, i} (\cF f)\\
       &\leq A \sup_{\l \in S_r} \int_{i\R{}} (1+|\l|^2)^{-\text{deg}\; p} \sum_{i=0}^{2k} |(\Delta+\rho^2)^{i} f(t)| (1+|\l|)^{\text{deg}\; p} (1+t)^{m+1}e^{(|\Re \l| -\rho)t} e^{2\rho t} dt \\
       &+A'\sum_{i=0}^{m-1} \omega_{k-m-i, i} (\cF f)<\infty
  \end{align*}   

\end{proof}

\begin{lemma}\label{lemma:Sp}
    Let $0 < r \leq 2$ with $r$ such that $\gamma_r$ is not a singularity of $\eE_w$. Then the operator $\cI_r : \mathscr{S}(S_r)_e \longrightarrow \mathcal{C}^r(X)^K$ defined as 
    \[\cI_r\phi (a_tw\cdot x_0):= 2\int_{\Re\lambda=- \gamma_r} \Phi_\l(t)[\phi(\l)]_w d\l = 2\int_{i\R{}} \Phi_{\l-\gamma_r}(t)[\phi(\l-\gamma_r)]_w d\l, \]
    is a continuous operator. 
\end{lemma}
\begin{proof}
 For $\phi \in \sS(S_r)_e$ we have 
  \bes
      \cJ \phi (a_tw\cdot x_0)= 2 \int_{i\R{}} \Phi_\l(t)[\phi(\l)]_w d\l.
  \ees
We recall that $L_r = \{\l \in L: \l  < \gamma_r\}$ and let $X_{\gamma_r} = \{n/2: n \in \bN_0\; \text{and}\; n/2 < \gamma_r\}$. Let $\eta$ be such that  $\max \{L_r \cup X_{\gamma_r}\} < \Re \eta < \gamma_r$. We note that such an $\eta$ exists as $\gamma_r$ is not a singularity of $\eE$. Therefore, $-\max\gamma_r<-\Re\eta<-\max L_r$. We now want to shift the integration (defined in $\cJ$) to $\Re\lambda=-\gamma_r$.  Using the Cauchy's formula we have 
  
  \begin{align*}
     \int_{-is}^{is} + \int_{is}^{is-\eta} - \int_{-is - \eta}^{is-\eta}- \int_{-is -\eta}^{-is} (\Phi_\l(t)[\phi(\l)]_w)  d\l = 2\pi i \sum_{\l_k \in L_r} \mr{Res}_{\l = -\l_k} \Phi_{\l}(t)[\phi(\l)]_w.
\end{align*}
Moreover, for a fixed positive integer $n$
\begin{align*}
    \lim_{s \rightarrow \infty}|\int_{is}^{is-\eta} (\Phi_\l(t)[\phi(\l)]_w)  d\l|&= \lim_{s \rightarrow \infty}|\int_{0}^{\eta} (\Phi_{-\beta + is}(t)[\phi(-\beta + is)]_w)d\beta| \\ &\leq \lim_{s \rightarrow \infty} M \int_{0}^\eta e^{(\beta-\rho)t} (1 + s^2)^{-n} d\beta\\
    & \leq \lim_{s \rightarrow \infty} M M_1 (1+s^2)^{-n}= 0,
\end{align*}

where we use the decay of $\phi$ and the estimate (\ref{est phi}). Since, $\Phi_\l$ does not have a pole on the line $[is, is-\eta]$, we have $|\Phi_\l(t)| \leq M e^{(|\Re \l|-\rho)t}$.

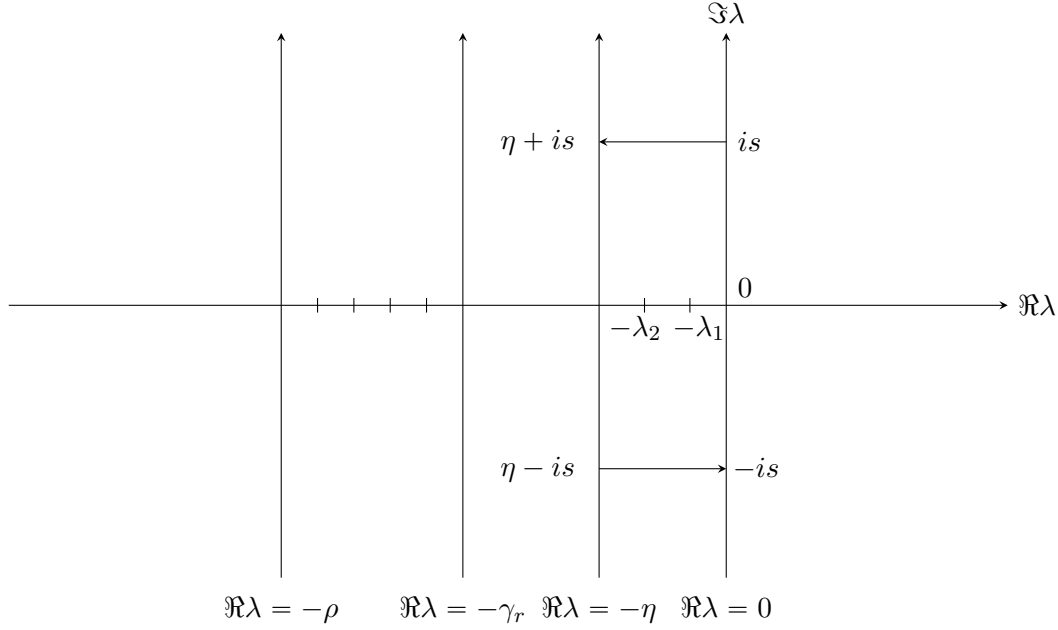
\begin{tikzpicture}[>=stealth, scale=1.2]

\draw[->] (-3,0) -- (8,0) node[right] {$\Re\lambda$};

\draw[->] (0,-3) -- (0,3) ;
\node[left] at (5.3,.2) {$0$};
\node[left] at (5.4,1.8) {$is$};
\node[left] at (5.6,-1.8) {$-is$};
\node[below] at (0,-3.1) {$\Re\lambda=-\rho$};

\draw[->] (2,-3) -- (2,3);
\node[below] at (2,-3.1) {$\Re\lambda=-\gamma_r$};
\node[right] at (2.3,1.8) {$\eta+is$};
\node[right] at (2.3,-1.8) {$\eta-is$};

\draw[->] (3.5,-3) -- (3.5,3);
\node[below] at (3.5,-3.1) {$\Re\lambda=-\eta$};

\draw[->] (4.9,-3) -- (4.9,3)node[above] {$\Im\lambda$};
\node[below] at (4.9,-3.1) {$\Re\lambda=0$};

\draw[->] (4.9,1.8) -- (3.5,1.8);
\draw[<-] (4.9,-1.8) -- (3.5,-1.8);

\node[below] at (4.6,0) {$-\lambda_1$};
\node[below] at (3.9,0) {$-\lambda_2$};

\foreach \x in {1.2,1.6,.8, .4, 4.5,4.0}
    \draw (\x,0.08) -- (\x,-0.08);

\end{tikzpicture}

\vspace{.7cm}

Similarly, we can show that \bes \int_{-is -\eta}^{-is} (\Phi_\l(t)[\phi(\l)]_w) d\l \rightarrow 0.\ees Thus, we can shift the integral to $\Re\lambda= - \eta$ and write
\[\cJ \phi (a_tw\cdot x_0) = 2 \int_{i\R{}} \Phi_\l(t)[\phi(\l)]_w d\l= 2\int_{i\R{}} \Phi_{\l-\eta}(t)[\phi(\l-\eta)]_w d\l + 4\pi i \sum_{\l_k \in L_r} \mr{Res}_{\l = -\l_k} \Phi_{\l}(t)[\phi(\l)]_w. \]
Since $\phi$ and all its derivatives extend continuously on the boundary of $S_r$ and also satisfies the decay estimates (3) of the definition on the boundary of $S_r$, 

\begin{align*}
    \int_{i\R{}} | \Phi_{\l-\eta}(t)[\phi(\l-\eta)]_w| d\l &\leq M \int_{i\R{}}e^{(|\Re \eta|-\rho)t } (1+|\l|^2)^{-n} d\l\\
    & \leq M M_1e^{(\gamma_r - \rho)t} 
\end{align*}
Thus, by Lebesgue Dominated convergence theorem the operator
\[\cI_r \phi(a_tw\cdot x_0) = 2\lim_{\Re \eta \rightarrow \gamma_r}\int_{i\R{}}\Phi_{\l-\eta}(t)[\phi(\l-\eta)]_w d\l\]
is well defined. Using the fact that \bes \mr{Res}_{\l = -\l_k}\Phi_{\l}(t) [\phi(\l)]_w = \mr{Res}_{\l = -\l_k}\Phi_{-\l}^0(t)[\phi(\l)]_w = \Phi_{-\l_K}^0(t)\mr{Res}_{\l = -\l_k}[\phi(\l)]_w, \ees we obtain
\begin{equation}\label{eq: IpI}
   \cJ \phi (a_t w) = \cI_r \phi (a_t w) + 4\pi i \sum_{\l_k \in L_r}\Phi_{-\l_k}^0(t) \mr{Res}_{\l = -\l_k} [\phi(\l)]_w. 
\end{equation}

Also we have $\Phi_{-\l_k} \in \mathcal{C}^2(X)^K$. From Theorem~\ref{thm:S2}, we also have $\cJ \phi \in \mathcal{C}^2(X)^K$. Thus, $\cI_r \phi \in \mathcal{C}^2(X)^K$. It remains to prove that $\cI_r \phi$ is in $\mathcal{C}^p(X)^K$. By expanding $\Phi_\l$, the function $\cI_r \phi$ can be written as
\[\cI_r \phi(a_tw) = 2e^{(-\gamma_r -\rho)t}\sum_{m \in \bN_0}e^{-mt} \int_{i\R{}} \Gamma_m(\l-\gamma_r) e^{\l t}[\phi(\l-\gamma_r)]_w d\l.\]
Using the properties of Euclidean Fourier transform $\cF_{\mr{euc}}$ and for a fixed $\delta >0$ we get
\begin{align*}
    \tau_{n_1,n_2}(\cI_r \phi) &\leq M \sup_{t \geq \delta}(1+t)^{n_1}e^{(2/p )\rho t}\left|\frac{\partial^{n_2}}{\partial t^{n_2}} \cI_r \phi(a_t w) \right|\\
    & \leq M \sup_{t \geq \delta}\left| \sum_{m\in \bN_0}e^{-mt} (1+t)^{n_1} \frac{\partial^{n_2}}{\partial t^{n_2}}\int_{\R{}}\Gamma_{m}(i\l-\gamma_r)[\phi(i\l-\gamma_r)]_w e^{i\l t}   d\l \right|\\
    &\leq M \sup_{t \geq \delta}\left| \sum_{m\in \bN_0}e^{-mt}(1+t)^{{n_1}} \frac{\partial^{n_2}}{\partial t^{n_2}} \cF_{\mr{euc}}(\Gamma_m(i\l-\gamma_r)[\phi(i\l-\gamma_r)]_w) \right|\\
    &\leq M \sup_{t \geq \delta} \sum_{m\in \bN_0}e^{-mt} \left|\cF_{\mr{euc}}\left((i\l-\gamma_r)^{n_2}\frac{\partial^{n_1}}{\partial \l^{n_1}}(\Gamma_m(i\l-\gamma_r)[\phi(i\l-\gamma_r)]_w) \right)\right|
    \end{align*}

By assumption $\phi$ has no poles on $\Re \l = -\gamma_r$ and $\|\phi^{(n)}(\l)\| \leq M (1+|\l|)^{-k}$ for large enough $k$ and any $n$ on $\Re \l = -\gamma_r$. Furthermore, $\Gamma_m$ has no poles on $\Re \l = -\gamma_r$ by Lemma~\ref{lemma:gamma}, using Cauchy's integral formula we get 
\[|\Gamma_m^{(n)}(i\lambda -\gamma_r)| \leq M (1+m)^\chi, \qquad \l \in \R{}.\]
    Thus,
    \begin{align*}
        \tau_{n_1,n_2}(\cI_r \phi)  & \leq M \sup_{t \geq \delta} \sum_{m\in \bN_0}e^{-mt}\int_{i\R{}}(1+|\l-\gamma_r|)^{n_2+2}(1+m)^{\chi} \sum_{i=0}^{n_1}\left\|\frac{\partial^i}{\partial \l^i}\phi(\l-\gamma_r)\right\| (1+|\l-\gamma_r|)^{-2} d\l \\
    & \leq M_\delta \sup_{\l\in \Re \l = -\gamma_r} (1+|\l|)^{n_2+2}\sum_{i=0}^{n_1}\left|\left|\frac{\partial^i}{\partial \l^i}\phi(\l)\right|\right| \int_{i\R{}}(1+|\l-\gamma_r|)^{-2} d\l<\infty.
\end{align*}
 Here, $M_\delta = M \sum_{m \in \bN_0}e^{-m\delta}(1+m)^\chi < \infty$. We thus conclude from (\ref{defn:schwartz}) that $\cI_r \phi$ is in $\mathcal{C}^r(X)^K$. 
\end{proof}

Consider the polynomial 
\[\pi(\l) = \prod_{\l_k \in L} (\l-\l_k^2+\rho^2).\]
Due to the fact that $\cF (\Delta f)(\l) = (\l^2-\rho^2)\cF f(\l)$, we have $\cF (\pi(\Delta) f)(\l) = \pi(\l^2 -\rho^2)\cF (f) (\l) $. 

\begin{proposition}\label{prop:qIpI}
    We have 
    \[\pi(\l^2-\rho^2)\cF \cI_r \phi (\l) = \pi(\l^2-\rho^2) \cF \cJ \phi (\l) = \pi(\l^2-\rho^2)\phi(\l) \]
\end{proposition}
\begin{proof}
We first note that \bes \Delta \Phi_{-\l_k}^0(t) = (\l_k^2-\rho^2)\Phi_{-\l_k}^0(t),\ees   for $\l_k \in L_r$. Thus, $\pi(\Delta) \Phi_{-\l_k}^0(t) = 0$ for all $\l_k \in L_p$. From (\ref{eq: IpI}) we have
    \begin{align*}
      \pi(\l^2-\rho^2) \cF \cJ \phi (\l) &= \cF \left[\pi(\Delta) \cI_r \phi (\l) + 4\pi i \sum_{\l_k \in L_p} \mr{Res}_{\l = -\l_k}  [\phi(\l)]_w \pi(\Delta) \Phi_{-\l_k}^0(t)\right]\\
      &= \pi(\l^2-\rho^2)\cF \cI_r \phi (\l). 
    \end{align*}
Since $\phi \in \sS(S_r)_e \subseteq \sS(i\R{})_e$. Thus, from part (3) of Theorem~\ref{thm:S2} we conclude
\[\pi(\l^2-\rho^2)\cF \cI_r \phi (\l) = \pi(\l^2-\rho^2) \cF \cJ \phi (\l) = \pi(\l^2-\rho^2)\phi(\l).\]
\end{proof}

\begin{lemma}
    The Fourier transform $\cF: \mathscr{S}^p(X)^K \longrightarrow \mathscr{S}(S_r)_e$ is a surjective linear  map. 
\end{lemma}
\begin{proof}
    Let $\phi \in \mathscr{S}(S_r)_e$. From Lemma~\ref{lemma:Sp} we have $\cI_r \phi \in \mathcal{C}^p(X)^K$. We will show that $\cF (\cI_r \phi) = \phi$. It follows from Proposition~\ref{prop:qIpI} that 
    \[\pi(\l^2-\rho^2)\cF \cI_r \phi (\l) = \pi(\l^2-\rho^2)\cF \cJ \phi (\l) = \pi(\l^2-\rho^2)\phi(\l).\]
    By dividing the polynomial, we obtain $\cF \cI_r \phi = \phi$ everywhere on $S_r$ except at the poles $\l_k \in L_r$ and $\cF \cI_r \phi$ and $\phi$ have simple poles at $\l_k$. Thus, 
    the surjectivity of the Fourier transform follows.
\end{proof}

\begin{lemma} We have the following:
\begin{enumerate}
   \item Let $1\leq r\leq 2$. \begin{enumerate}
        \item Suppose that $L$ is non empty set. Then the kernel of $\cF: \mathcal{C}^r(X)^K \longrightarrow \sS(S_r{})_e$ is the span of the functions $\{\Phi^0_{-\l_k}(t): \l_k \in L_r^c\}$. 
        \item If $L$ is an empty set then $\cF$ is an injective linear  map .
    \end{enumerate}
    
    \item If $0<r<1$, then $\cF$ is an injective linear  map. 
\end{enumerate}
    
\end{lemma}
\begin{proof}
    Assume $1\leq r\leq 2$. We observed that $\l_k\in L_r^c$ implies $\Phi^0_{-\l_k}$ is in $\mathcal{C}^r(X)^K$. Also we proved in (\ref{ft:discrete}) that $\cF(\Phi^0_{-\l_k})(\l)=0$ for all $\l\in i\mathbb R$. Let $f\in \mathcal{C}^r(X)^K$ be in the kernel of $\cF$. Therefore  by the inversion formula (\ref{inversion-1}) we have \begin{equation}\label{eq:1}
        f(a_tw\cdot x_0) = \sum_{\l_k \in L}c_{\l_k}^w \Phi_{-\l_k}^0 (t),
    \end{equation}
    for some $c_{\l_k}^w\in \C$. Now let $\l'=\min\{\l: \l\in L\}$. Choose an $r'$ with $r\leq r'<2$ such that $\l'<\gamma_{r'}$ and $\l>\gamma_{r'}$ for all $\l\in L, \l\not=\l'$. Then we rewrite (\ref{eq:1}) as
    \begin{equation}
        f(a_tw\cdot x_0) = c_{\l'}^w\Phi_{-\l'}^0(t) + \sum_{\substack{\l_k \in L,\\ \l_k\not=\l'}}c_{\l_k}^w \Phi_{-\l_k}^0 (t),
    \end{equation}
    Since $\l>\gamma_{r'}$ for all $\l\in L, \l\not=\l'$, the function \bes \sum_{\substack{\l_k \in L,\\ \l_k\not=\l'}}c_{\l_k}^w \Phi_{-\l_k}^0 (t),\ees is in $\sS^{r'}(X)^K$. Also $f\in \mathcal{C}^r(X)^K$ implies  $f\in \mathcal C^{r'}(X)^K$. Hence, $c_{\l'}^w \Phi_{-\l'}^0(t)\in \mathcal C^{r'}(X)^K$. But $\l'<\gamma_{r'}$ implies $\Phi_{-\l'}^0\not\in \mathcal C^{r'}(X)^K$. Hence $c_{\l'}^w=0$. We repeat the process and will get that \begin{equation}
        f(a_tw\cdot x_0) = \sum_{\l_k \in L_r^c}c^w_{\l_k} \Phi_{-\l_k}^0 (t).
    \end{equation}
    This completes the proof. If $L$ is an empty set we use the inversion formula to imply that $\cF$ is an injective linear  map. 
    
    In the case $0<r<1$, we use the same technique to prove that all the coefficients $c_{\l_k}^w = 0$ for all $\l_k \in L$. This implies that $f(a_tw\cdot x_0) = 0$. 
\end{proof}

 Let $f \in \mathcal{C}^r(X)^K, 0 < r\leq 2$. Then we recall the following inversion formula:
    \[f(a_tw\cdot x_0) =  \int_{i\R{}} \eE_w(\lambda, \cF{ f}(\lambda))(t) d\lambda + 4\pi i\sum_{\l_k \in L} \Phi_{-\l_k}^0 (t) \mr{Res}_{\l = -\l_k}[\cF{f} (\l)]_w.\]

    Now we define (following Barker \cite{Ba88})
    \bes
f_H(a_tw\cdot x_0) =  \int_{i\R{}} \eE_w(\lambda, \cF{ f}(\lambda))(t) d\lambda + 4\pi i\sum_{\l_k \in L_r} \Phi_{-\l_k}^0 (t) \mr{Res}_{\l = -\l_k}[\cF{f} (\l)]_w.
    \ees
     and 
\bes
f_B(a_tw\cdot x_0) =   4\pi i\sum_{\l_k \in L_r^c} \Phi_{-\l_k}^0 (t) \mr{Res}_{\l = -\l_k}[\cF{f} (\l)]_w.
    \ees
    Then it follows that \bes
f=f_H + f_B,
    \ees
    and this decomposition is unique.
    We define $\mathcal C^r_H(X)^K$ as 
    \bes
\mathcal C^r_H(X)^K=\{f_H: f\in \mathcal C^r(X)^K \}.
    \ees
    Then from the previous discussion we have the following corollary:

    \begin{corollary}
      The Fourier transform $\cF: \mathcal{C}_H^r(X)^K \longrightarrow \mathscr{S}(S_r)_e$ is a continuous isomorphism.  
    \end{corollary}
    


\bibliographystyle{plain} 
\bibliography{ref}

\end{document}